\documentclass[11pt,twoside,reqno]{amsart}

\usepackage{amssymb, latexsym}
\usepackage[italian,english]{babel}
\usepackage{amsmath,amsfonts,amsthm}
\usepackage{paralist}
\usepackage[top=2.50cm, bottom=2.50cm, left=2.50cm, right=2.50cm]{geometry}

\usepackage{color}
\numberwithin{equation}{section}
\usepackage{xspace}
\usepackage[bookmarksnumbered,colorlinks]{hyperref}
\usepackage{graphics}
\def\bb#1\eb{\textcolor{blue}
{#1}} %
\def\br#1\er{\textcolor{red}
{#1}} %
\def\bv#1\ev{\textcolor{green}
{#1}} %
\def\bc#1\ec{\textcolor{cyan}
{#1}} %

\usepackage{graphics}

\def\Xint#1{\mathchoice
  {\XXint\displaystyle\textstyle{#1}}%
  {\XXint\textstyle\scriptstyle{#1}}%
  {\XXint\scriptstyle\scriptscriptstyle{#1}}%
  {\XXint\scriptscriptstyle\scriptscriptstyle{#1}}%
  \!\int}
\def\XXint#1#2#3{{\setbox0=\hbox{$#1{#2#3}{\int}$}
  \vcenter{\hbox{$#2#3$}}\kern-.5\wd0}}
\def\-int{\Xint -}

\newcommand{\R}{\mathbb{R}}
\newcommand{\N}{\mathbb{N}}

\newcommand{\X}{\mathbb{X}_{0}}
\newcommand{\C}{\mathcal{C}}
\newcommand{\A}{\mathbb{H}^{s}(\Omega)}
\newcommand{\Y}{H^{1}_{0, L}(y^{1-2s})}

\DeclareMathOperator{\dive}{div}

\DeclareMathOperator{\e}{\varepsilon}

\newtheorem{lem}{Lemma}[section]
\newtheorem{thm}{Theorem}[section]
\newtheorem{defn}{Definition}[section]

\newtheorem{remark}{Remark}[section]

\begin{document}
\title[Multiple solutions for superlinear fractional problems via theorems of mixed type]{Multiple solutions for superlinear fractional problems via theorems of mixed type} 

\author[V. Ambrosio]{Vincenzo Ambrosio}
\address[V. Ambrosio]{Dipartimento di Scienze Pure e Applicate (DiSPeA),
Universit\`a degli Studi di Urbino `Carlo Bo'
Piazza della Repubblica, 13
61029 Urbino (Pesaro e Urbino, Italy)}
\email{vincenzo.ambrosio@uniurb.it}


\begin{abstract}
In this paper we investigate the existence of multiple solutions for the following two fractional problems
\begin{equation*}
\left\{
\begin{array}{ll}
(-\Delta_{\Omega})^{s} u-\lambda u= f(x, u) &\mbox{ in } \Omega\\
u=0 &\mbox{ in } \partial \Omega
\end{array}
\right.
\end{equation*}
and
\begin{equation*}
\left\{
\begin{array}{ll}
(-\Delta_{\R^{N}})^{s} u-\lambda u= f(x, u) &\mbox{ in } \Omega\\
u=0 &\mbox{ in } \R^{N}\setminus \Omega,
\end{array}
\right.
\end{equation*}
where $s\in (0,1)$, $N>2s$, $\Omega$ is a smooth bounded domain  of $\R^{N}$, and  $f:\bar{\Omega}\times \R\rightarrow \R$ is a superlinear continuous function which does not satisfy the well-known Ambrosetti-Rabinowitz condition. Here $(-\Delta_{\Omega})^{s}$ is the spectral Laplacian and $(-\Delta_{\R^{N}})^{s}$ is the fractional Laplacian in $\R^{N}$. By applying variational theorems of mixed type due to Marino and Saccon and the Linking Theorem, we prove the existence of multiple solutions for the above problems.
\end{abstract}

\maketitle
\section{Introduction}
\noindent
In this paper we focus our attention on the multiplicity of the following two fractional problems
\begin{equation}\label{P1}
\left\{
\begin{array}{ll}
(-\Delta_{\Omega})^{s} u-\lambda u= f(x, u) &\mbox{ in } \Omega\\
u=0 &\mbox{ in } \partial \Omega
\end{array}
\right.
\end{equation}
and
\begin{equation}\label{P}
\left\{
\begin{array}{ll}
(-\Delta_{\R^{N}})^{s} u-\lambda u= f(x, u) &\mbox{ in } \Omega\\
u=0 &\mbox{ in } \R^{N}\setminus \Omega,
\end{array}
\right.
\end{equation}
where $s\in (0, 1)$, $N>2s$, $\lambda \in \R$ and $\Omega\subset \R^{N}$ is a smooth bounded open set. 
Here $(-\Delta_{\Omega})^{s}$ is the spectral Laplacian which is given by 
\begin{align}
(-\Delta_{\Omega})^{s}u(x)= \sum_{k=1}^{+\infty} c_{k} \alpha_{k}^{s} \varphi_{k}(x) \mbox{ for any } u=\sum_{k=1}^{+\infty} c_{k}\varphi_{k}\in C^{\infty}_{c}(\Omega),
\end{align}
where $\{\varphi_{k}\}_{k\in \N}$ denotes the orthonormal basis of $L^{2}(\Omega)$ consisting of eigenfunctions of $-\Delta$ in $\Omega$ with homogeneous Dirichlet boundary conditions associated to the eigenvalues $\{\alpha_{k}\}_{k\in \N}$, that is,
\begin{equation*}
\left\{
\begin{array}{ll}
-\Delta \varphi_{k}=\alpha_{k} \varphi_{k} &\mbox{ in } \Omega\\
\varphi_{k}=0 &\mbox{ in } \partial \Omega.
\end{array}
\right.
\end{equation*}
The fractional Laplacian operator $(-\Delta_{\R^{N}})^{s}$ may be defined for any $u:\R^{N}\rightarrow \R$ belonging to the Schwarz space $\mathcal{S}(\R^{N})$ of rapidly decaying $C^{\infty}$ functions in $\R^{N}$ by 
\begin{align}
(-\Delta_{\R^{N}})^{s} u(x)= \frac{C_{N, s}}{2} \int_{\R^{N}} \frac{2u(x)- u(x+y)- u(x-y)}{|y|^{N+2s}} dy,
\end{align}
where $C_{N, s}$ is a normalizing constant depending only on $N$ and $s$; see \cite{DPV, MBRS} for more details. \\
As observed in \cite{sv3}, these two operators are completely different. Indeed, the spectral operator  $(-\Delta_{\Omega})^{s}$ depends on the domain $\Omega$ considered, while the integral one $(-\Delta_{\R^{N}})^{s}$ evaluated at some point is independent on the domain in which the equation is set. Moreover, in contrast with the setting for the fractional Laplacian, it is not true that all functions are s-harmonic with respect to the spectral fractional Laplacian, up to a small error; see \cite{AV, DSV} for more details.\\
Recently, many papers have appeared dealing with the existence and the multiplicity of solutions to problems driven by these two operators, by applying several variational and topological techniques. In particular, a great attention has been devoted to the study of fractional problems like \eqref{P1} and \eqref{P} involving superlinear nonlinearities with subcritical or critical growth; see for instance \cite{A1, A2, A3, A4, A5, AFP, BCdPS, BMBS, CT, MBRS, PPS, sv1, sv2, sv4, YYY}.
It is worth observing that a typical assumption to study this class of problems is to require that the nonlinearity $f$ verifies the well-known Ambrosetti-Rabinowitz condition \cite{AR}, that is there exist $\mu>2$ and $R>0$ such that
\begin{align}\label{AR}
0< \mu F(x, t) \leq t f(x, t) \mbox{ for any } x\in \Omega, |t|>R.
\end{align}
This condition is quite natural and fundamental not only to guarantee that the Euler-Lagrange functional associated to the problem under consideration has a mountain pass geometry, but also to show that the Palais-Smale sequence of the Euler-Lagrange functional is bounded.
We recall that \eqref{AR} is somewhat restrictive and eliminates many nonlinearities. For instance the function 
\begin{equation}\label{exf}
f(x, t)=2t\log(1+t^{4})+\frac{4t^{5}}{t^{4}+1} \mbox{ with } (x, t)\in \Omega\times \R
\end{equation}
 is superlinear at infinity but does not verify the condition \eqref{AR}.\\
The purpose of this paper is to investigate the multiplicity for the above two fractional problems when the parameter $\lambda$ lies in a suitable neighborhood of any eigenvalue of the fractional operator under consideration, and $f$ is superlinear and subcritical, but does not fulfill \eqref{AR}. \\
More precisely, along the paper we assume that $f:\bar{\Omega}\times \R\rightarrow \R$ is a continuous function satisfying the following conditions
\begin{compactenum}[$(f1)$]
\item there exist $c_{1}>0$ and $p\in (1, 2^{*}_{s}-1)$, with $2^{*}_{s}=\frac{2N}{N-2s}$, such that
$$
|f(x, t)|\leq c_{1}(1+|t|^{p}) \mbox{ for any } (x, t)\in \Omega\times \R;
$$
\item 
$$
\lim_{|t|\rightarrow 0} \frac{f(x, t)}{|t|}=0 \mbox{ uniformly in } x\in \Omega;
$$
\item 
$$
\lim_{|t|\rightarrow \infty} \frac{F(x, t)}{t^{2}}=+\infty  \mbox{ uniformly in } x\in \Omega, 
$$
where $F(x, t)=\int_{0}^{t} f(x, \tau)\, d\tau$;
\item 
there exist $\beta\in (\frac{2Np}{N+2s}, 2^{*}_{s})$, $c_{2}>0$ and $T>0$ such that
\begin{align*}
&f(x, t)t-2F(x, t)>0 \mbox{ for any } x\in \Omega, |t|>0,\\
&f(x, t)t-2F(x, t)\geq c_{2}|t|^{\beta} \mbox{ for any } x\in \Omega, |t|\geq T;
\end{align*}
\item 
$F(x, t)\geq 0$ for any $(x, t)\in \Omega\times \R$.
\end{compactenum}
As a model for $f$ we can take the function defined in \eqref{exf}.
\noindent
Now we state our first main result regarding the multiplicity for the problem \eqref{P1}:
\begin{thm}\label{thm2}
Assume $(f1)$-$(f5)$. Then for any $i\geq 2$ there exists $\delta_{i}>0$ such that for any $\lambda\in (\alpha^{s}_{i}-\delta_{i}, \alpha^{s}_{i})$, problem \eqref{P1} admits at least three nontrivial solutions. 
\end{thm}

\noindent
In order to prove Theorem \ref{thm2}, we apply suitable variational methods after transforming the problem \eqref{P1} into a degenerate elliptic equation with a nonlinear Neumann boundary condition by using the extension technique \cite{BrCDPS, CT, CS, CDDS}. 
Thanks to this approach we are able to overcome the nonlocality of the operator $(-\Delta_{\Omega})^{s}$ and we can use some critical point results to study the extended problem.
More precisely, we show that the functional associated to the extended problem respects the geometry required by the $\nabla$-Theorem introduced by Marino and Saccon in \cite{MS}.
Roughly speaking, this theorem says that if a $C^{1}$-functional $I$ defined on a Hilbert space has a linking structure and $\nabla I$ verifies an appropriate condition on some suitable sets (see Definition 2.1 below), then $I$ has two nontrivial critical points which may have the same critical level. We will apply this abstract result to the functional associated to the extended problem and we will get the existence of two nontrivial solutions. Finally, exploiting an additional linking structure, we get the existence of a third nontrivial solution. 
We recall that in the local setting, similar arguments have been developed and applied in many situations to obtain multiplicity results for several and different problems such as, elliptic problems of second and fourth order, noncooperative elliptic systems, nonlinear Schr\"odinger equations with indefinite linear part in $\R^{N}$, variational inequalities; see \cite{MMS, MS2, OL, w, WZZ}. 
Differently from the classic case, in the nonlocal framework, the only result comparable to Theorem \ref{thm2} is due to Mugnai and Pagliardini \cite{MP} who obtained a multiplicity result to problem \eqref{P1} when $s= \frac{1}{2}$ and $f$ satisfies \eqref{AR}.

\noindent
Our second main result concerns the multiplicity of solutions to \eqref{P}. 
\begin{thm}\label{thm1}
Assume $(f1)$-$(f5)$. Then for any $i\geq 2$ there exists $\delta_{i}>0$ such that for any $\lambda\in (\lambda_{i}-\delta_{i}, \lambda_{i})$, problem \eqref{P} admits at least three nontrivial solutions. Here $\{\lambda_{k}\}_{k\in \N}$ are the eigenvalues of the fractional Laplacian $(-\Delta_{\R^{N}})^{s}$ with homogeneous condition in $\R^{N}\setminus \Omega$.
\end{thm}

\noindent
The proof of the above result is obtained following the approach developed to prove Theorem \ref{thm2}. 
Anyway, we do not make use of any extension method and our techniques work also when we replace $(-\Delta_{\R^{N}})^{s}$ by the more general integro-differential operator $-\mathcal{L}_{K}$ defined up to a positive constant as
$$
\mathcal{L}_{\mathcal{K}} u(x)=\int_{\R^{N}}  (u(x+y)+u(x-y)-2u(x)) \mathcal{K}(y) dy \quad (x\in \R^{N}),
$$ 
where $\mathcal{K}:\R^{N}\setminus\{0\}\rightarrow (0, \infty)$ is a measurable function such that $\mathcal{K}(-x)=\mathcal{K}(x)$ for all $x\in \R^{N}\setminus\{0\}$, $m \mathcal{K}\in L^{1}(\R^{N})$ with $m(x)=\min\{|x|^{2}, 1\}$, and there exists $\theta>0$ such that $\mathcal{K}(x)\geq \theta |x|^{-(N+2s)}$ for all $x\in \R^{N}\setminus\{0\}$.

In this context, we take care of the well-known results on the spectrum of integro-differential operators obtained by Servadei and Valdinoci in \cite{sv2, sv3}. We point out that in a recent paper Molica Bisci et al. \cite{MBMS} proved a similar result to Theorem \ref{thm1} when $f$ verifies condition \eqref{AR}, obtaining a nonlocal counterpart of the multiplicity result established in \cite{M}. 

The paper is organized as follows. In Section $2$ we recall some useful results related to the extension method in a bounded domain and then we provide some useful lemmas which will be fundamental to apply a critical point theorem of mixed nature. In Section $3$ we deal with the existence of three nontrivial weak solutions to the problem \eqref{P}.

\section{multiplicity for the problem \eqref{P1}}
\subsection{Extended problem in the half-cylinder}
In order to study problem \eqref{P1}, we use a suitable variant of the extension technique due to Caffarelli and Silvestre \cite{CS}; see \cite{BrCDPS, CT, CDDS} for more details. Firstly, we collect some useful notations and basic results which will be useful along the paper. \\
Fix $s\in (0, 1)$.
We say that $u\in H^{s}(\Omega)$ if $u\in L^{2}(\Omega)$ and it holds
$$
[u]_{H^{s}(\Omega)}^{2}=\iint_{\Omega\times \Omega} \frac{|u(x)-u(y)|^{2}}{|x-y|^{N+2s}} \, dx dy<\infty.
$$
We define $H^{s}_{0}(\Omega)$ as the closure of $C^{\infty}_{c}(\Omega)$ with respect to the norm $[u]_{H^{s}(\Omega)}^{2}+\|u\|_{L^{2}(\Omega)}^{2}$. The space $H^{\frac{1}{2}}_{00}(\Omega)$ is the Lions-Magenes   space  \cite{LM}  which consists of the function $u\in H^{\frac{1}{2}}(\Omega)$ such that
$$
\int_{\Omega} \frac{u^{2}(x)}{{\rm dist}(x, \partial \Omega)}\, dx<\infty.
$$
Let us introduce the Hilbert space
$$
\A=\Bigl\{u\in L^{2}(\Omega): \sum_{k=1}^{\infty} |c_{k}|^{2}\alpha_{k}^{s}<\infty\Bigr\}.
$$
It is well known \cite{LM} that interpolation leads to 
\begin{equation*}
\A=
\left\{
\begin{array}{ll}
H^{s}(\Omega) &\mbox{ if } s\in (0, \frac{1}{2})\\
H^{\frac{1}{2}}_{00}(\Omega) &\mbox{ if } s=\frac{1}{2} \\
H^{s}_{0}(\Omega) &\mbox{ if } s\in (\frac{1}{2}, 1).
\end{array}
\right.
\end{equation*}
Let us define the cylinder $\C=\Omega\times (0, +\infty)$ and its lateral boundary $\partial_{L}\C=\partial \Omega\times [0, +\infty)$.
Let us denote by  $H^{1}_{0, L}(y^{1-2s})$ the space of measurable functions $v:\C\rightarrow \R$ such that $v\in H^{1}(\Omega\times (\alpha, \beta))$ for all $0<\alpha<\beta<+\infty$, $u=0$ on  $\partial_{L}\C$ and for which the following norm is finite
$$
\|u\|_{\Y}^{2}=\iint_{\C} y^{1-2s} |\nabla u|^{2}\, dx dy.
$$
We recall the following trace theorem which relates $\Y$ to $\A$.
\begin{thm}\cite{BrCDPS, CT, CDDS}\label{tracethm}
There exists a surjective continuous linear map $\textup{Tr}: \Y\rightarrow \A$ such that, for any $u\in \Y$
$$
\kappa_{s} \|\textup{Tr}(u)\|^{2}_{\A}\leq \|u\|^{2}_{\Y}.
$$
\end{thm}
\noindent
We also have some useful embedding results.
\begin{thm}\cite{BrCDPS, CT, CDDS}\label{SSembedding}
Let $N> 2s$ and $q\in [1, 2^{*}_{s}]$. Then there exists a constant $C$ depending on $N$, $q$ and the measure of $\Omega$, such that, for all $u\in \X$
$$
\|\textup{Tr}(u)\|_{L^{q}(\Omega)}\leq C \|u\|_{\Y}.
$$
Moreover, $\Y$ is compactly embedded into $L^{q}(\Omega)$ for any $q\in [1, 2^{*}_{s})$.
\end{thm}

\noindent
Thus, we get the following fundamental result which allows us to realize the fractional spectral Laplacian $(-\Delta_{\Omega})^{s}$.
\begin{thm}\cite{BrCDPS, CT, CDDS}
Let $u\in \A$. Then there exists a unique $v\in \Y$ such that 
\begin{equation*}
\left\{
\begin{array}{ll}
-\dive(y^{1-2s} v)= 0 &\mbox{ in } \C \\
v=0 &\mbox{ on } \partial_{L} \C\\
\textup{Tr}(v)=u &\mbox{ on } \partial^{0} \C.
\end{array}
\right.
\end{equation*}
Moreover
$$
\frac{\partial v}{\partial \nu^{1-2s}}:=-\lim_{y\rightarrow 0^{+}} y^{1-2s} v_{y}(x, y)=\kappa_{s}(-\Delta_{\Omega})^{s} u(x) \mbox{ in } \A^{*},  
$$
where $\A^{*}$ is the dual of $\A$. The function $v$ is called the extension of $u$.
\end{thm}

\noindent 
We also recall that if $u=\sum_{k=1}^{\infty} c_{k} \varphi_{k}\in \A$, then the extension of $u$ is given by 
$$
v(x, y)=\sum_{k=1}^{\infty} c_{k} \varphi_{k}(x) \theta(\sqrt{\lambda_{k}} y),
$$
where $\theta\in H^{1}(\R_{+}, y^{1-2s})$ solves the problem 
\begin{equation*}
\left\{
\begin{array}{ll}
\theta''+\frac{1-2s}{y}\theta'-\theta= 0 &\mbox{ in } \R_{+}\\
\theta(0)=1, \quad \mbox{ and } -\lim_{y\rightarrow 0^{+}} y^{1-2s}\theta'(y)=\kappa_{s}.
\end{array}
\right.
\end{equation*}
In addition, $\|v\|_{\Y}^{2}=\kappa_{s}\|u\|^{2}_{\A}$; see \cite{BrCDPS, CT, CDDS} for more details. 
\begin{remark}
In order to simplify notation, when no confusion arises, we shall denote by $v$ the function defined in the cylinder $\C$ as well as its trace $\textup{Tr}(v)$ on $\Omega\times\{y=0\}$.
\end{remark} 
Taking into account the above results, we can deduce that the study of \eqref{P1} is equivalent to consider the following degenerate elliptic problem with nonlinear Neumann boundary condition
\begin{equation}\label{P3}
\left\{
\begin{array}{ll}
\dive(y^{1-2s}\nabla u)=0 &\mbox{ in } \C\\
u=0 &\mbox{ on } \partial_{L} \C \\
\frac{\partial u}{\partial \nu^{1-2s}}=\kappa_{s}[\lambda u+f(x, u)] &\mbox{ on } \partial \Omega\times \{0\}.
\end{array}
\right.
\end{equation}
For simplicity, in what follows, we will assume that $\kappa_{s}=1$.

\subsection{Technical lemmas and $\nabla$-condition}

For $i\geq 2$, let us introduce the following notations. 
Let $H_{i}={\rm span} \{\psi_{1}, \dots, \psi_{i}\}$,  
where $\psi_{k}(x, y)=\varphi_{k}(x)\theta(\sqrt{\lambda_{k}}y)$, 
$$
H_{i}^{\perp}=\{u\in \Y: \langle u, \psi_{j}\rangle_{\Y}=0, \mbox{ for all } j=1, \dots, i\}
$$ 
and $H_{i}^{0}=\{\psi_{i}, \dots, \psi_{j} \}$.
We set $\mu_{i}=\alpha_{i}^{s}$.
Since $\{\alpha_{k}\}_{k\in \N}$ is increasing, a direct calculation yields the next result.
\begin{lem}\label{Slem}
For any $i\geq 1$, the following inequalities hold
\begin{equation}\label{7s}
\|u\|_{\Y}^{2}\leq \mu_{i} \|u\|_{L^{2}(\Omega)}^{2} \mbox{ for any } u\in H_{i}
\end{equation}
and
\begin{equation}\label{8s}
\|u\|_{\Y}^{2}\geq \mu_{i+1} \|u\|_{L^{2}(\Omega)}^{2} \mbox{ for any } u\in H_{i}^{\perp}.
\end{equation}
\end{lem}

\noindent
Now, we prove an auxiliary lemma.
\begin{lem}\label{clem}
Let $K: L^{2}(\Omega)\rightarrow \Y$ be the operator defined by setting $K(u)=v$, where $v$ is the unique solution to 
\begin{equation}\label{PK}
\left\{
\begin{array}{ll}
-\dive(y^{1-2s} v)= 0 &\mbox{ in } \C \\
v=0 &\mbox{ on } \partial_{L} \C\\
\frac{\partial v}{\partial \nu^{1-2s}}=u &\mbox{ on } \partial^{0} \C.
\end{array}
\right.
\end{equation}
Then $K$ is compact.
\end{lem}
\begin{proof}
Let $\{u_{n}\}_{n\in \N}$ be a bounded sequence in $L^{2}(\Omega)$. From the weak formulation of \eqref{PK} and Theorem \ref{SSembedding}, we can see that
$$
\|v_{n}\|_{\Y}^{2}\leq \|u_{n}\|_{L^{2}(\Omega)}\|v_{n}\|_{L^{2}(\Omega)}\leq C \|u_{n}\|_{L^{2}(\Omega)} \|v_{n}\|_{\Y}
$$ 
that is, $\{v_{n}\}_{n\in \N}$ is bounded in $\Y$. Then, in view of Theorem \ref{SSembedding}, we may assume that 
\begin{align}\begin{split}\label{SCS}
&v_{n}\rightharpoonup v \mbox{ in } \Y \\
&v_{n}\rightarrow v \mbox{ in } L^{q}(\Omega) \quad \forall q\in [1, 2^{*}_{s}).
\end{split}\end{align}
Now, by using \eqref{PK}, we can see that for any $n\in \N$
\begin{equation}\label{waffle}
\|v_{n}\|_{\Y}^{2}-\iint_{\C} y^{1-2s} \nabla v_{n}\nabla v\, dx dy=\int_{\Omega} u_{n}(v_{n}-v)\, dx.
\end{equation}
Taking into account \eqref{SCS} and the fact that $\{u_{n}\}_{n\in \N}$ is bounded in $L^{2}(\Omega)$, from \eqref{waffle} we can deduce that $\|v_{n}\|_{\Y}\rightarrow \|v\|_{\Y}$. Since $\Y$ is a Hilbert space, we can conclude the proof.
\end{proof}

\noindent
In order to deduce a multiplicity result for \eqref{P3} we need to recall the $\nabla$-Theorem due to Marino and Saccon \cite{MS}.
We begin giving the following definition.
\begin{defn}
Let $X$ be a Hilbert space, $I\in C^{1}(X, \R)$ and $M$ a closed subspace of $X$, $a, b\in \R\cup \{-\infty, +\infty\}$. We say that the condition $(\nabla) (I, M, a, b)$ holds if there is $\gamma>0$ such that
$$
\inf \{\|P_{M} \nabla I(u)\|: a\leq I(u)\leq b, dist(u, M)<\gamma\}>0,
$$
where $P_{M}: X\rightarrow M$ is the orthogonal projection of $X$ onto $M$.
\end{defn}

\noindent
Therefore, if the above condition holds, then $I_{\mid M}$ has no critical points $u$ such that $a\leq I(u)\leq b$, with some uniformity.
\begin{thm}\cite{MS}\label{MS}
Let $X$ be a Hilbert space and $X_{i}$, $i=1, 2, 3$ three subspaces of $X$ such that $X=X_{1}\oplus X_{2}\oplus X_{3}$ and $dim(X_{i})<+\infty$ with $i=1, 2$. Let us denote by $P_{i}$ the orthogonal projection of $X$ onto $X_{i}$, $I\in C^{1, 1}(X, \R)$. Let $R, R', R'', \varrho>0$ such that $R'<\varrho<R''$. Define
\begin{align*}
&\Gamma=\{u\in X_{1}\oplus X_{2}: R'\leq \|u\|\leq R'', \|P_{1}u\|\leq R\} \mbox{ and } T=\partial_{X_{1}\oplus X_{2}} \Gamma, \\
&S_{23}(\varrho)=\{u\in X_{2}\oplus X_{3}: \|u\|=\varrho\} \mbox{ and } B_{23}(\varrho)=\{u\in X_{2}\oplus X_{3}: \|u\|\leq \varrho\}.
\end{align*}
Especially, if $R'=0$, $T$ may be defined as follows:
\begin{align*}
T=&\{u\in X_{1}: \|u\|\leq R\}\cup \{u\in X_{1}\oplus X_{2}: \|P_{2}u\|=R'', \|P_{1}u\|\leq R\} \\
&\cup \{u\in X_{1}\oplus X_{2}: \|P_{2}u\|\leq R'', \|P_{1}u\|= R\}.
\end{align*}
Assume that
$$
a'=\sup I(T)<\inf I(S_{23}(\varrho))=a''.
$$
Let $a$ and $b$ such that $a'<a<a''$ and $b>\sup I(\Gamma)$. Assume that $(\nabla) (I, X_{1}\oplus X_{3}, a, b)$ holds and that the $(PS)_{c}$ condition holds at any $c\in [a, b]$. Then $I$ has at least two critical points in $I^{-1}([a, b])$. Moreover, if
$$
\inf I(B_{23}(\varrho))>a_{1}>-\infty
$$
and the $(PS)_{c}$ condition holds at any $c\in [a_{1}, b]$, then $I$ has another critical level in $[a_{1}, a']$. 
\end{thm}

\noindent
Now, we introduce the energy functional $I_{\lambda}: \Y\rightarrow \R$ associated to \eqref{P3}, that is,
$$
I(u)=\frac{1}{2}\|u\|^{2}_{\Y}-\frac{\lambda}{2}\int_{\Omega} u^{2}\, dx-\int_{\Omega} F(x, u)\, dx
$$
defined for any $u\in \Y$. From the assumptions on $f$, it is clear that the functional $I_{\lambda}$ is well-defined,  $I_{\lambda}\in C^{1}(\Y, \R)$ and its derivative is given by
$$
\langle I'_{\lambda}(u), v\rangle=\iint_{\C} y^{1-2s} \nabla u \nabla v \, dx dy-\lambda \int_{\Omega} u v \, dx-\int_{\Omega} f(x, u) v\, dx \mbox{ for any } v\in \Y.
$$
Since we aim to show that $I_{\lambda}$ verifies the assumptions of Theorem \ref{MS}, we need to prove some useful lemmas which allow us to verify that there exist $0<a<b$ such that the condition $(\nabla)(I_{\lambda}, H_{i-1}\oplus H_{i}^{\perp}, a, b)$ holds.

\begin{lem}\label{lem1s}
Assume that $(f1)$ and $(f4)$ hold. Then, for any $\delta\in (0, \min\{\mu_{i+1}-\mu_{i}, \mu_{i}-\mu_{i-1}\})$ there exists $\e_{0}>0$ such that for any $\lambda\in [\mu_{i}-\delta, \mu_{i}+\delta]$ the unique critical point $u$ of $I_{\lambda}$ constrained on $H_{i-1}\oplus H_{i}^{\perp}$ such that $I_{\lambda}(u)\in [-\e_{0}, \e_{0}]$ is the trivial one.
\end{lem}

\begin{proof}
Suppose by contradiction that there exist $\delta_{0}$, $\lambda_{n}\in [\mu_{i}- \delta_{0}, \mu_{i}+ \delta_{0}]$ and $\{u_{n}\}_{n\in \N} \subset H_{i-1} \oplus H_{i}^{\perp}\setminus\{0\}$ such that, for any $v\in H_{i-1}\oplus H_{i}^{\perp}$ we get
\begin{align}
&I_{\lambda_{n}}(u_{n})=\frac{1}{2} \iint_{\C} y^{1-2s} |\nabla u_{n}|^{2} \, dx dy-\frac{\lambda_{n}}{2}\int_{\Omega} |u_{n}|^{2} dx - \int_{\Omega} F(x, u_{n})\, dx 
\rightarrow 0 \label{9s}\\
&\langle I_{\lambda_{n}}'(u_{n}), v\rangle=\iint_{\C} y^{1-2s} \nabla u_{n} \nabla v \, dx dy- \lambda_{n}\int_{\Omega} u_{n}\, v \, dx - \int_{\Omega} f(x, u_{n}) \,v\,dx
=0. \label{10s}
\end{align}
Up to a subsequence, we may assume that $\lambda_{n}\rightarrow \lambda \in [\mu_{i}- \delta_{0}, \mu_{i}+ \delta_{0}]$ as $n\rightarrow \infty$. Let us define $A_{n}:=\{x \in \Omega : |u_{n}(x)|\geq T\}$. Then, by assumption $(f4)$ we deduce 
\begin{align}\label{tv1s}
2 I_{\lambda_{n}}(u_{n})- \langle I_{\lambda_{n}}'(u_{n}), u_{n}\rangle = \int_{\Omega}(f(x, u_{n}) \, u_{n} - 2F(x, u_{n}))\, dx \geq c_{2} \int_{A_{n}} |u_{n}|^{\beta} \, dx. 
\end{align}
By using \eqref{9s} and \eqref{10s} with $v=u_{n}$, from inequality \eqref{tv1s} we get
\begin{align}\label{11s}
\int_{A_{n}} |u_{n}|^{\beta} \, dx \rightarrow 0 \, \mbox{ as } n\rightarrow \infty. 
\end{align}
Now, let us observe that 
\begin{align}\label{12s}
\int_{\Omega} |u_{n}|^{\beta} \, dx = \int_{A_{n}} |u_{n}|^{\beta} \, dx + \int_{\Omega \setminus A_{n}}|u_{n}|^{\beta} \, dx \leq \int_{A_{n}} |u_{n}|^{\beta}\, dx + |\Omega| T^{\beta}. 
\end{align}
Set $u_{n}= v_{n}+ w_{n} \in H_{i-1}\oplus H_{i}^{\perp}$. Then, by using \eqref{7s}, \eqref{8s}, the fact that $\|u_{n}\|_{\Y}^{2}=\|v_{n}\|_{\Y}^{2}+ \|w_{n}\|_{\Y}^{2}$ and \eqref{10s}, we have
\begin{align}\label{13s}
&\int_{\Omega} f(x, u_{n})(w_{n}- v_{n})\, dx\nonumber \\
&= \|w_{n}\|_{\Y}^{2} - \lambda_{n}\int_{\Omega} |w_{n}|^{2} dx - \|v_{n}\|_{\Y}^{2} + \lambda_{n} \int_{\Omega} |v_{n}|^{2} \, dx\nonumber \\
&\geq \frac{\mu_{i+1} - \lambda_{n}}{\mu_{i+1}}\|w_{n}\|_{\Y}^{2} - \frac{\mu_{i-1} - \lambda_{n}}{\mu_{i-1}}\|v_{n}\|_{\Y}^{2}\nonumber \\
&\geq c_{3}\|u_{n}\|_{\Y}^{2}, 
\end{align}
where $c_{3}= \min \left\{\frac{\mu_{i+1}-\lambda_{n}}{\mu_{i+1}}, \frac{\lambda_{n}-\mu_{i-1}}{\mu_{i-1}}\right\}$.\\
From Theorem \ref{SSembedding} and by applying H\"older's inequality, we can infer that
\begin{align}\label{tv2s}
\int_{\Omega} f(x, u_{n})(w_{n}- v_{n})\, dx & \leq \left( \int_{\Omega} |f(x, u_{n})|^{\frac{p+1}{p}}dx \right)^{\frac{p}{p+1}} \left( \int_{\Omega} |w_{n}- v_{n}|^{p+1} dx\right)^{\frac{1}{p+1}} \nonumber \\
%&\leq 2C \|u_{n}\| \left(\int_{\Omega} |f(x, u_{n})|^{\frac{p+1}{p}}dx \right)^{\frac{p}{p+1}}\nonumber \\
& \leq 2C \|u_{n}\|_{\Y} \left(\int_{\Omega} |f(x, u_{n})|^{\frac{p+1}{p}}dx \right)^{\frac{p}{p+1}}.
\end{align}
Taking into account \eqref{13s} and \eqref{tv2s}, and recalling that $u_{n}\not \equiv 0$, we have
\begin{align}\label{14s}
\|u_{n}\|_{\Y}\leq c_{4} \left( \int_{\Omega} |f(x, u_{n})|^{\frac{p+1}{p}}dx \right)^{\frac{p}{p+1}}, 
\end{align}
for some positive constant  $c_{4}$. \\
Now, by using $(f1)$, Theorem \ref{SSembedding}, \eqref{12s} and H\"older's inequality we can deduce that
\begin{align}\label{tv3s}
\left| \int_{\Omega} f(x, u_{n})(v_{n}- w_{n})\, dx \right| &\leq\int_{\Omega} |f(x, u_{n})| |v_{n}- w_{n}|\, dx \nonumber \\
&\leq c_{1} \int_{\Omega} (|v_{n} - w_{n}|+ |u_{n}|^{p} |v_{n}- w_{n}|) \, dx\nonumber \\
&\leq c_{1} \|v_{n}- w_{n}\|_{L^{1}(\Omega)}+ c_{1} \left( \int_{\Omega} |u_{n}|^{\beta} dx \right)^{\frac{p}{\beta}}\left(\int_{\Omega} |v_{n}- w_{n}|^{\frac{\beta}{\beta-p}}dx \right)^{\frac{\beta-p}{\beta}} \nonumber \\
&\leq c_{1} C \|v_{n}- w_{n}\|_{\Y} \left( 1+ \left( \int_{A_{n}} |u_{n}|^{\beta} dx + |\Omega|T^{\beta}\right)^{\frac{p}{\beta}}\right)\nonumber \\
&\leq c_{1} C\|u_{n}\|_{\Y} \left( 1+ \left( \int_{A_{n}} |u_{n}|^{\beta} + |\Omega| T^{\beta}\right)^{\frac{p}{\beta}}\right). 
\end{align}
Therefore, putting together \eqref{11s}, \eqref{13s} and \eqref{tv3s}, we deduce that $\{u_{n}\}_{n\in \N}$ is bounded in $\Y$. Hence, in view of Theorem \ref{SSembedding}, we may assume that, up to a subsequence, there are a sequence $\{u_{n}\}_{n\in \N}$ and a function $u\in \Y$ such that 
\begin{align}\begin{split}\label{limitss}
&u_{n}\rightharpoonup u \mbox{ in } \Y \\
&u_{n}\rightarrow u \mbox{ in } L^{r}(\Omega) \mbox{ for all } r\in [1, 2^{*}_{s})\\
&u_{n}(x) \rightarrow u(x) \mbox{ a.e. } x\in \Omega. 
\end{split}\end{align}
By applying \eqref{9s}, \eqref{10s} and Fatou's Lemma we get
\begin{align*}
0&= \lim_{n\rightarrow \infty} 2I_{\lambda_{n}}(u_{n}) - \langle I_{\lambda_{n}}'(u_{n}), u_{n}\rangle \\
&= \lim_{n\rightarrow \infty} \int_{\Omega} (f(x, u_{n}) u_{n} - 2F(x, u_{n}))\, dx \\
&\geq \int_{\Omega} \liminf_{n\rightarrow \infty} (f(x, u_{n})u_{n} - 2 F(x, u_{n}))\, dx\\
&= \int_{\Omega} (f(x,u)u- 2F(x,u))\, dx
\end{align*}
which, combined with the assumptions $(f2)$ and $(f4)$, gives $u=0$. \\
Now, we distinguish two cases. Let us assume that $u_{n}\rightarrow 0$ as $n\rightarrow \infty$ in $H^{1}_{0,L}(y^{1-2s})$. From $(f1)$ and $(f2)$ we know that for any $\e>0$ there exists $c_{\e}>0$ such that
\begin{equation}\label{mbmss}
|f(x, t)|\leq \e |t|+c_{\e}|t|^{p} \mbox{ for any } (x, t)\in \Omega\times \R.
\end{equation} 
By using \eqref{14s}, \eqref{mbmss} and Theorem \ref{SSembedding}, we have
\begin{align*}
1\leq \lim_{n\rightarrow \infty} c_{4} \frac{\left(\int_{\Omega} |f(x, u_{n})|^{\frac{p+1}{p}}dx \right)^{\frac{p}{p+1}}}{\|u_{n}\|_{\Y}}=0.
%\leq \lim_{n\rightarrow \infty} c_{4}'(\e+\|u_{n}\|_{\Y}^{p-1}) =c_{4}'\e.
\end{align*}
%From the arbitrariness of $\e$ we get a contradiction. \\
On the other hand, if there exists $\alpha>0$ such that $\|u_{n}\|_{\Y}\geq \alpha$ for $n$ large enough, then from \eqref{14s}, \eqref{limitss}, $(f2)$, the Dominated Convergence Theorem and $u=0$, we get
\begin{align*}
0<\alpha \leq \lim_{n\rightarrow \infty} c_{4} \left( \int_{\Omega} |f(x, u_{n})|^{\frac{p+1}{p}}dx \right)^{\frac{p}{p+1}}=0, 
\end{align*}
which is a contradiction. This ends the proof of the lemma.
\end{proof}

\begin{lem}\label{lem2s}
Assume that $(f1)$ and $(f4)$ hold, $\lambda\in (\mu_{i-1}, \mu_{i+1})$ and $\{u_{n}\}_{n\in \N}\subset \Y$ such that $I_{\lambda}(u_{n})$ is bounded, $Pu_{n}\rightarrow 0$ and $Q \nabla I_{\lambda}(u_{n})\rightarrow 0$ as $n\rightarrow +\infty$. Then $\{u_{n}\}_{n\in \N}$ is bounded in $\Y$.
\end{lem}

\begin{proof}
Assume by contradiction that, up to a subsequence, $\|u_{n}\|_{\Y}\rightarrow \infty$ as $n\rightarrow \infty$. \\
Note that $u_{n}=Pu_{n}+Qu_{n}$, $Pu_{n}\rightarrow 0$ in $\Y$ and $Q\nabla I_{\lambda}(u_{n})\rightarrow 0$, where $\nabla I_{\lambda}(u_{n})= v_{n}$ is such that 
\begin{equation*}
\langle I'_{\lambda}(u_{n}), z\rangle= \iint_{\C} y^{1-2s} v_{n} \nabla z \, dxdy
\end{equation*}
for any $z\in \Y$. So we get
\begin{equation*}
v_{n}= u_{n}- K(\lambda u_{n} + f(x, u_{n})),
\end{equation*}
where $K$ is defined as in Lemma \ref{clem}.
Now, we recall that $u_{n}=Pu_{n}+Qu_{n}$ and $Pu_{n}\rightarrow 0$ in $\Y$. Then, by exploiting the assumption $(f1)$, H\"older's inequality and the fact that all norms in $H_{i}^{0}$ are equivalent, we can see that
\begin{align}\label{tv4s}
\left|\int_{\Omega} f(x, u_{n}) Pu_{n}\, dx \right|&\leq \int_{\Omega} |f(x, u_{n})| |Pu_{n}|\, dx \nonumber \\
&\leq c_{1} \left( \int_{\Omega} |Pu_{n}|\, dx + \int_{\Omega} |Pu_{n}| |u_{n}|^{p}\, dx \right) \nonumber \\
&\leq c_{1} \|Pu_{n}\|_{L^{1}(\Omega)} + c_{1} \left( \int_{\Omega} |u_{n}|^{\beta} dx \right)^{\frac{p}{\beta}}\left(\int_{\Omega} |Pu_{n}|^{\frac{\beta}{\beta-p}}dx \right)^{\frac{\beta-p}{\beta}}\nonumber \\
&\leq c_{5} \|Pu_{n}\|_{L^{\infty}(\Omega)} (1+ \|u_{n}\|_{L^{\beta}(\Omega)}^{p}), 
\end{align}
with $c_{5}>0$. 
Now, from the assumption $(f4)$ and \eqref{tv4s}, we can deduce that
\begin{align}\label{tv5s}
&2I_{\lambda}(u_{n})- \langle Q\nabla I_{\lambda}(u_{n}), u_{n}\rangle \nonumber \\
&= \int_{\Omega} (f(x, u_{n}) u_{n} - 2F(x, u_{n}))\, dx + \|Pu_{n}\|_{\Y}^{2} - \lambda \int_{\Omega} |Pu_{n}|^{2} dx - \int_{\Omega} f(x, u_{n}) Pu_{n} \, dx \nonumber \\
&\geq c_{2} \|u_{n}\|_{L^{\beta}(\Omega)}^{\beta} + \|Pu_{n}\|_{\Y}^{2} -\lambda \|Pu_{n}\|_{L^{2}(\Omega)}^{2} - c_{5} \|Pu_{n}\|_{L^{\infty}(\Omega)} (1+ \|u_{n}\|_{L^{\beta}(\Omega)}^{p}).
\end{align}
Here we used that for every $z\in \Y$, $Pz$ is smooth and $\nabla Pu_{n}=P\nabla u_{n}$ due to $u\in {\rm span}\{\psi_{i}, \dots, \psi_{j}\}$ and $Pz\perp Qz$, so we have
\begin{align*}
\iint_{\C} y^{1-2s} &\nabla (P(u_{n}-K(\lambda u_{n}+g(x, u_{n})))) \nabla u_{n} \, dx dy\\
&= \|Pu_{n}\|_{\Y}^{2} - \lambda \int_{\Omega} |Pu_{n}|^{2} dx - \int_{\Omega} f(x, u_{n}) Pu_{n} \, dx.
\end{align*}
Since $1<p<\beta$, ${\rm dim}H_{i}^{0}<+\infty$ and $\|Pu_{n}\|_{L^{\infty}(\Omega)} \rightarrow 0$ as $n\rightarrow \infty$, from \eqref{tv5s} we can infer that
\begin{align}\label{tv60s}
\frac{\|u_{n}\|_{L^{\beta}(\Omega)}^{p}}{\|u_{n}\|_{\Y}}\rightarrow 0 \mbox{ as } n\rightarrow \infty. 
\end{align}
Set $Qu_{n}= v_{n}+ w_{n} \in H_{i-1}\oplus H_{i}^{\perp}$. By using $(f1)$, Theorem \ref{SSembedding}, \eqref{7s} and H\"older's inequality we have
\begin{align*}
\langle Q\nabla I_{\lambda}(u_{n}), -v_{n} \rangle &= \lambda \|v_{n}\|_{L^{2}(\Omega)}^{2} - \|v_{n}\|_{\Y}^{2} + \int_{\Omega} f(x, u_{n})v_{n}\, dx\\
&\geq \frac{\lambda- \mu_{i-1}}{\mu_{i-1}} \|v_{n}\|^{2} - \int_{\Omega} |f(x, u_{n})||v_{n}|\, dx \\
&\geq \frac{\lambda- \mu_{i-1}}{\mu_{i-1}}\|v_{n}\|_{\Y}^{2} -c_{1} \int_{\Omega} (|u_{n}|^{p} |v_{n}| + |v_{n}|)\, dx \\
&\geq \frac{\lambda- \mu_{i-1}}{\mu_{i-1}}\|v_{n}\|_{\Y}^{2} -c_{1} \left(\int_{\Omega} |u_{n}|^{\beta}dx\right)^{\frac{p}{\beta}} \left(\int_{\Omega} |v_{n}|^{\frac{\beta}{\beta-p}} dx\right)^{\frac{\beta-p}{\beta}} - c_{1}\|v_{n}\|_{L^{1}(\Omega)} \\
&\geq \frac{\lambda- \mu_{i-1}}{\mu_{i-1}}\|v_{n}\|_{\Y}^{2} -c'_{1} C\|v_{n}\|_{\Y} (1+ \|u_{n}\|_{L^{\beta}(\Omega)}^{p}). 
\end{align*}
Therefore, \eqref{tv60s} and H\"older's inequality imply that
\begin{align}\label{16s}
\frac{\|v_{n}\|_{\Y}}{\|u_{n}\|_{\Y}}\rightarrow 0 \mbox{ as } n\rightarrow \infty. 
\end{align}
In similar fashion we can infer that
\begin{align}\label{17s}
\frac{\|w_{n}\|_{\Y}}{\|u_{n}\|_{\Y}}\rightarrow 0 \mbox{ as } n\rightarrow \infty. 
\end{align}
We can also show that 
\begin{align}\label{18s}
\frac{\|Pu_{n}\|_{\Y}}{\|u_{n}\|_{\Y}}\rightarrow 0 \mbox{ as } n\rightarrow \infty. 
\end{align}
Indeed, if \eqref{18s} does not hold, then $\frac{\|Pu_{n}\|_{\Y}}{\|u_{n}\|_{\Y}}\rightarrow \ell\in (0, +\infty)$ and we can see that 
$$
0\leftarrow \|Pu_{n}\|_{\Y}=\frac{\|Pu_{n}\|_{\Y}}{\|u_{n}\|_{\Y}}\|u_{n}\|_{\Y}\rightarrow \ell \cdot (+\infty)=+\infty
$$
which is impossible.
Putting together \eqref{16s}, \eqref{17s} and \eqref{18s} we deduce that
\begin{align*}
1= \frac{\|u_{n}\|_{\Y}}{\|u_{n}\|_{\Y}}\leq \frac{\|v_{n}\|_{\Y}+ \|Pu_{n}\|_{\Y} + \|w_{n}\|_{\Y}}{\|u_{n}\|_{\Y}}\rightarrow 0 \mbox{ as } n\rightarrow \infty, 
\end{align*}
which is a contradiction. Thus $\{u_{n}\}_{n\in \N}$ is bounded in $\Y$. 
\end{proof}

\begin{lem}\label{lem3s}
Assume $(f1)$ and $(f4)$. Then, for any $\delta\in (0, \min\{\mu_{i+1}-\mu_{i}, \mu_{i}-\mu_{i-1}\})$ there exists $\e_{0}>0$ such that for any $\lambda\in [\mu_{i}-\delta, \mu_{i}+\delta]$ and for any $\e_{1}, \e_{2}\in (0, \e_{0})$ with $\e_{1}<\e_{2}$, the condition $(\nabla) (I_{\lambda}, H_{i-1}\oplus H_{i}^{\perp}, \e_{1}, \e_{2})$ holds.
\end{lem}

\begin{proof}
Suppose by contradiction that there exists a positive constant $\delta_{0}$ such that for all $\varepsilon_{0}>0$ there are $\lambda\in [\mu_{i}- \delta_{0}, \mu_{i}+ \delta_{0}]$ and $\varepsilon_{1}, \varepsilon_{2} \in (0, \varepsilon_{0})$ with $\varepsilon_{1}< \varepsilon_{2}$ such that the condition $(\nabla)(I_{\lambda}, H_{i-1}\oplus H_{i}^{\perp}, \varepsilon_{1}, \varepsilon_{2})$ does not hold. \\
Let $\varepsilon_{0}>0$ be as in Lemma \ref{lem1s}. Then, we can find a sequence $\{u_{n}\}_{n\in \N}\subset \Y$ such that ${\rm dist}(u_{n}, H_{i-1} \oplus H_{i}^{\perp})\rightarrow 0$, $I_{\lambda}(u_{n})\in (\varepsilon_{1}, \varepsilon_{2})$ and $Q\nabla I_{\lambda}(U_{n})\rightarrow 0$. By Lemma \ref{lem2s} we deduce that $\{u_{n}\}_{n\in \N}$ is bounded. Thus, by applying Theorem \ref{SSembedding}, there are a subsequence (still denoted by $u_{n}$) and $u\in \Y$ such that $u_{n}\rightharpoonup u$ in $\Y$ and $u_{n}\rightarrow u$ in $L^{q}(\Omega)$ for any $q\in [1, 2^{*}_{s})$. Taking into account assumption $(f1)$, $Q\nabla I_{\lambda}(u_{n})\rightarrow 0$, $Pu_{n}\rightarrow 0$ and Lemma \ref{clem}, we can see that 
$$
Q\nabla I_{\lambda}(u_{n})=u_{n}-Pu_{n}-K(\lambda u_{n}+f(x, u_{n}))
$$
yields $u_{n}\rightarrow u$ in $\Y$ and $u$ is a critical point of $I_{\lambda}$ constrained on $H_{i-1}\oplus H_{i}^{\perp}$. Hence, in view of  Lemma \ref{lem1s}, we can infer that $u=0$. Since $0<\varepsilon_{1}\leq I_{\lambda}(u)$, we obtain a contradiction.  
\end{proof}

\noindent
Let us introduce the following notations: for fixed $i, k\in \N$ and $R, \varrho >0$, let
\begin{align*}
&B_{i}(R)= \{u\in H_{i} : \|u\|_{\Y}\leq R\}, \\
&T_{i-1, i}(R) = \{u \in H_{i-1}: \|u\|_{\Y}\leq R\} \cup \{u \in H_{i} : \|u\|_{\Y}=R\},  \\
&S_{k}^{+}(\varrho)= \{u \in H_{k}^{\perp}: \|u\|_{\Y}= \varrho\}, \\
&B_{k}^{+}(\varrho)= \{u\in H_{k}^{\perp}: \|u\|_{\Y}\leq \varrho\}.
\end{align*}

\begin{lem}\label{lem4s}
Assume that $(f1)$-$(f3)$ and $(f5)$. Then, for any $\lambda\in (\mu_{i-1}, \mu_{i+1})$, there are $R>\varrho>0$ such that
$$
0=\sup I_{\lambda}(T_{i-1, i}(R))<\inf I_{\lambda}(S_{i-1}^{+}(\varrho)).
$$
\end{lem}

\begin{proof}
By using \eqref{7s} and the assumption $(f5)$, for any $u\in H_{i-1}$ and $\lambda\in (\mu_{i-1}, \mu_{i})$ we have
\begin{align}\label{19s}
I_{\lambda}(u)&= \frac{1}{2} \|u\|_{\Y}^{2} - \frac{\lambda}{2} \int_{\Omega} |u|^{2} dx - \int_{\Omega} F(x, u)\, dx \nonumber \\
&\leq \frac{\mu_{i-1}- \lambda}{2\mu_{i-1}}\|u\|_{\Y}^{2} \leq 0. 
\end{align}
Taking into account the assumption $(f3)$ and the continuity of $F$, for any $c_{6}>0$ there is $M_{1}>0$ such that 
\begin{align}\label{tv6s}
F(x, t) \geq \frac{c_{6}}{2}t^{2} - M_{1} \quad \forall (x, t)\in \Omega \times \R.
\end{align}
By using \eqref{7s} and \eqref{tv6s}, for any $u\in H_{i}$ and $\lambda\in (\mu_{i-1}, \mu_{i})$ we have
\begin{align}\label{tv7s}
I_{\lambda}(u)&= \frac{1}{2} \|u\|_{\Y}^{2} - \frac{\lambda}{2} \int_{\Omega} |u|^{2} dx - \int_{\Omega} F(x, u)\, dx  \nonumber \\
&\leq \frac{\mu_{i}- \lambda}{2\mu_{i}} \|u\|_{\Y}^{2} - \frac{c_{6}}{2} \|u\|_{L^{2}(\Omega)}^{2} + M_{1}|\Omega| \nonumber \\
& \leq \frac{\mu_{i}- \lambda- c_{6}}{2\mu_{i}} \|u\|_{\Y}^{2} + M_{1}|\Omega|. 
\end{align}
Taking $c_{6}= 2(\mu_{i}- \lambda)$, from \eqref{tv7s} we deduce that 
\begin{align}\label{20s}
I_{\lambda}(u)\rightarrow - \infty \mbox{ as } \|u\|_{\Y}\rightarrow \infty. 
\end{align}
Now, we note that $(f1)$ and $(f2)$ imply that for any $\e>0$ there is $C_{\e}>0$ such that 
\begin{align}\label{tv8s}
F(x, t)\leq \frac{\e}{2} t^{2} + C_{\e} |t|^{p+1} \quad \forall (x, t)\in \Omega \times \R,
\end{align}
which gives
\begin{align}\label{21s}
\left| \int_{\Omega} F(x, u)\, dx \right| \leq \frac{\e}{2} \|u\|_{L^{2}(\Omega)}^{2} + C_{\e} \|u\|_{L^{p+1}(\Omega)}^{p+1} \quad \forall u\in \Y. 
\end{align}
Thus, from \eqref{21s},  we can see that for any $u\in H_{i-1}^{\perp}$
\begin{align}\label{22s}
I_{\lambda}(u)&= \frac{1}{2}\|u\|_{\Y}^{2} - \frac{\lambda}{2} \int_{\Omega} |u|^{2} dx - \int_{\Omega} F(x, u)\, dx  \nonumber \\
&\geq \frac{\mu_{i}- \lambda- \e}{2 \mu_{i}} \|u\|_{\Y}^{2} - CC_{\e} \|u\|_{\Y}^{p+1}. 
\end{align} 
Take $\e= \frac{\mu_{i}- \lambda}{2}>0$. Recalling that $\lambda \in (\mu_{i-1}, \mu_{i})$ and $p+1>2$, from \eqref{19s}, \eqref{20s} and \eqref{22s}, we can find $R>\varrho >0$ such that
\begin{align*}
\sup I_{\lambda} (T_{i-1, i}(R))< \inf I_{\lambda} (S_{i-1}^{+}(\varrho)).
\end{align*}
\end{proof}

\begin{lem}\label{lem5s}
Assume that $(f5)$ holds. Then, for $R>0$ in Lemma \ref{lem4s} and for any $\e>0$ there exists $\delta'_{i}>0$ such that for any $\lambda\in (\mu_{i}-\delta'_{i}, \mu_{i})$ it holds
$$
\sup I_{\lambda}(B_{i}(R))<\e.
$$
\end{lem}

\begin{proof}
By using \eqref{7s}, the assumption $(f5)$ and $\lambda<\mu_{i}$, we deduce that, for any $u\in H_{i}$ we deduce
\begin{align*}
I_{\lambda}(u) = \frac{1}{2} \|u\|_{\Y}^{2} - \frac{\lambda}{2} \int_{\Omega} |u|^{2} dx - \int_{\Omega} F(x, u)\, dx \leq \frac{\mu_{i}- \lambda}{2 \mu_{i}} \|u\|_{\Y}^{2}. 
\end{align*}
Take $\delta_{i}'= \frac{\mu_{i}\e}{R^{2}}$. Then we deduce that 
\begin{align*}
\sup I_{\lambda}(B_{i}(R)) \leq \frac{\mu_{i}- \lambda}{2 \mu_{i}} R^{2} = \frac{(\mu_{i}- \lambda)\e}{2 \delta_{i}'}<\e. 
\end{align*}
\end{proof}

\begin{lem}\label{lem6s}
Assume that $(f1)$ and $(f4)$ hold. Then $I_{\lambda}$ verifies the Palais-Smale condition.
\end{lem}
\begin{proof}
Let $\{u_{n}\}_{n\in \N}$ be a Palais-Smale sequence of $I_{\lambda}$.  Taking into account $(f1)$, we have only to show that $\{u_{n}\}_{n\in \N}$ is bounded. From the arguments in Lemma \ref{lem2s}, it is enough to prove that 
\begin{equation}\label{tv8s}
\frac{\|Pu_{n}\|_{\Y}}{\|u_{n}\|_{\Y}}\rightarrow 0 \mbox{ as } n\rightarrow +\infty.
\end{equation}
In view of $(f4)$, we know that there exist $c_{7}, c_{8}>0$ such that 
\begin{align*}
f(x, t)t-2F(x, t)\geq c_{7}|t|-c_{8} \mbox{ for any } (x, t)\in \Omega\times \R.
\end{align*}
Then, by using the above inequality, the equivalence of the norms on the finite-dimensional space, and Theorem \ref{SSembedding}, we get
\begin{align}\label{tv9s}
2I_{\lambda}(u_{n})-\langle I'_{\lambda}(u_{n}), u_{n}\rangle&=\int_{\Omega} (f(x, u_{n})u_{n}-2F(x, u_{n})) \, dx \nonumber\\
&\geq \int_{\Omega} (c_{7}|u_{n}|-c_{8})\, dx \nonumber\\
&\geq \int_{\Omega} (c_{7}|Pu_{n}|-c_{7}|v_{n}|-c_{7}|w_{n}|-c_{8})\, dx \nonumber\\
&\geq c_{9}\|Pu_{n}\|_{L^{1}(\Omega)}-c_{10}(\|v_{n}\|_{\Y}+\|w_{n}\|_{\Y}+1) \nonumber \\
&\geq c'_{9}\|Pu_{n}\|_{\Y}-c_{10}(\|v_{n}\|_{\Y}+\|w_{n}\|_{\Y}+1).
\end{align} 
Putting together \eqref{16s}, \eqref{17s} and \eqref{tv9s} we can deduce that \eqref{tv8s} holds.
\end{proof}

\noindent
Now we are in the position to give the proof of the main result of this section.
\begin{proof}[Proof of Theorem \ref{thm2}]
Firstly, we prove the existence of two critical points.
Taking into account Lemma \ref{lem3s}, Lemma \ref{lem4s} and Lemma \ref{lem5s}, we can take $a\in (0, \inf I_{\lambda}(S^{+}_{i-1}(\varrho)))$ and $b>\sup I_{\lambda}(B_{i}(R))$ such that $0<a<b<\e_{0}$. Then the condition $(\nabla)(I_{\lambda}, H_{i-1}\oplus H_{i}^{\perp}, a, b)$ is satisfied.By applying Lemma \ref{lem6s} and Theorem \ref{MS}, we can deduce that there exist two critical points $u_{1}, u_{2}\in H^{1}_{0,L}(y^{1-2s})$ such that $I_{\lambda}(u_{i})\in [a, b]$ for $i=1, 2$.
Now, we prove the existence of a third critical point by invoking the Linking Theorem \cite{Rab}. Taking into account Theorem $5.3$ in \cite{Rab} and Lemma \ref{lem6s}, it is enough to prove that there are $\delta_{i}''>0$ and $R_{1}>\varrho_{1}>0$ such that for any $\lambda\in (\mu_{i}-\delta''_{i}, \mu_{i})$ we get
\begin{equation}\label{23s}
\sup I_{\lambda}(T_{i, i+1}(R_{1}))<\inf I_{\lambda}(S^{+}(\varrho_{1})).
\end{equation}  
Let us note that \eqref{8s}, \eqref{21s} and Theorem \ref{SSembedding} yield 
\begin{align}\label{tv10s}
I_{\lambda}(u)&=\frac{1}{2} \|u\|_{\Y}^{2}-\frac{\lambda}{2} \int_{\Omega} u^{2} \, dx-\int_{\Omega} F(x, u)\, dx \nonumber \\
&\geq \frac{\mu_{i+1}-\lambda-\e}{2\mu_{i+1}} \|u\|_{\Y}^{2}-CC_{\e}\|u\|_{\Y}^{p+1} \mbox{ for any } u\in H_{i}^{\perp}.
\end{align}
Take $\e=\frac{\mu_{i+1}-\lambda}{2}$. Then, recalling that $p>1$, in view of \eqref{tv10s}, we can find $\varrho_{1}>0$ and $\alpha>0$ such that
\begin{equation}\label{24s}
\inf I_{\lambda}(S^{+}_{i}(\varrho_{1}))\geq \alpha>0.
\end{equation}
Now, by using \eqref{7s} and $(f5)$, we deduce that
\begin{align}\label{tv11s}
I_{\lambda}(u)&=\frac{1}{2}\|u\|_{\Y}^{2}-\frac{\lambda}{2} \int_{\Omega} u^{2} \, dx-\int_{\Omega} F(x, u)\, dx \nonumber \\
&\leq \frac{\mu_{i}-\lambda}{2\mu_{i}} \|u\|_{\Y}^{2} \mbox{ for any } u\in H_{i}.
\end{align}
Hence, by using \eqref{tv11s}, we can see that there exist $\delta''_{i}>0$ and $R_{1}>0$ such that for any $\lambda\in (\mu_{i}-\delta''_{i}, \mu_{i})$ we get
\begin{equation}\label{25s}
I_{\lambda}(u)< \alpha \mbox{ for any } \|u\|_{\Y}\leq R_{1}.
\end{equation}
On the other hand, by using \eqref{7s} and $(f5)$, we can see that for any $u\in H_{i+1}$ and $\lambda\in (\mu_{i}-\delta''_{i}, \mu_{i})$, we have 
\begin{align}\label{26s}
I_{\lambda}(u)&=\frac{1}{2} \|u\|_{\Y}^{2}-\frac{\lambda}{2} \int_{\Omega} u^{2} \, dx-\int_{\Omega} F(x, u)\, dx \nonumber \\
&\leq \frac{\mu_{i+1}-\lambda}{2\mu_{i+1}} \|u\|_{\Y}^{2}.
\end{align}
Putting together \eqref{24s}, \eqref{25s} and \eqref{26s} we can infer that \eqref{23s} is verified. By applying the Linking Theorem, we can deduce that there exists a critical point $u_{3}\in \Y$ of $I_{\lambda}$ such that $I_{\lambda}(u)\geq \inf I_{\lambda}(S^{+}_{i}(\varrho_{1}))$.
Choosing $\delta_{i}=\min\{ \delta'_{i}, \delta''_{i} \}$, where $\delta'_{i}$ is given in Lemma \ref{lem5s}, we can conclude that Theorem \ref{thm2} holds.
\end{proof}

\section{multiple solutions for the problem \eqref{P}}
\noindent
This section is devoted to the proof of Theorem \ref{thm1}. Since many calculations are adaptations to the ones presented in the previous section, we will emphasize only the differences between the ``spectral'' and the ``integral'' case.
Firstly, we collect some notations and results which we will use in the sequel. For more details we refer the interested reader to \cite{MBRS, sv1, sv2, sv3}.\\
Let us define
$$
\X=\{u\in H^{s}(\R^{N}): u=0 \mbox{ a.e. in } \R^{N}\setminus \Omega\}.
$$
endowed wit the norm
$$
\|u\|_{\X}^{2}=\iint_{\R^{2N}} \frac{|u(x)-u(y)|^{2}}{|x-y|^{N+2s}} \, dx dy.
$$
Then, $\X$ is a Hilbert space, and the following useful embedding result holds.
\begin{thm}\cite{sv1}\label{Sembedding}
$\X$ is compactly embedded into $L^{q}(\R^{N})$ for any $q\in [1, 2^{*}_{s})$.
\end{thm}

\noindent
Let us denote by $\{e_{k}, \lambda_{k}\}_{k\in \N}$ the eigenvalues and corresponding eigenfunctions of the fractional Laplacian operator $(-\Delta_{\R^{N}})^{s}$ with homogeneous boundary condition in $\R^{N}\setminus \Omega$, that is,
\begin{equation*}
\left\{
\begin{array}{ll}
(-\Delta_{\R^{N}})^{s} e_{k}=\lambda_{k} e_{k} &\mbox{ in } \Omega\\
e_{k}=0 &\mbox{ in } \R^{N}\setminus \Omega.
\end{array}
\right.
\end{equation*}
We recall that $\lambda_{1}$ is simple, $0<\lambda_{1}<\lambda_{2}\leq \dots\leq \lambda_{k}\leq \lambda_{k+1}\leq \dots$, $\lambda_{k}\rightarrow +\infty$ and $e_{k}$ are H\"older continuous up to the boundary (differently from the ones of $(-\Delta_{\Omega})^{s}$ that are as smooth up the boundary as the boundary allows).

\noindent
As in Section $2$, for any $i\geq 2$, we denote by $P: \X\rightarrow H^{0}_{i}$ and $Q: \X\rightarrow H_{i-1}\oplus H_{i}^{\perp}$ the orthogonal projections, where $H_{i}^{0}={\rm span}\{e_{i}, \dots, e_{j}\}$.
The next lemma is proved in \cite{sv2}.
\begin{lem}\cite{sv2}\label{SVlem}
The following inequalities holds
\begin{align}
&\iint_{\R^{2N}} \frac{|u(x)- u(y)|^{2} dx}{|x-y|^{N+2s}} \, dxdy \leq \lambda_{j} \int_{\Omega} |u|^{2} dx \mbox{ for all } u\in H_{j} \label{7}, \\
&\iint_{\R^{2N}} \frac{|u(x)- u(y)|^{2} dx}{|x-y|^{N+2s}} \, dxdy \geq \lambda_{j+1} \int_{\Omega} |u|^{2} dx \mbox{ for all } u\in H_{j}^{\perp} \label{8}. \\
\end{align}
\end{lem}

\noindent
We say that a a function $u\in \X$ is a weak solution to \eqref{P} if it satisfies the identity
$$
\iint_{\R^{2N}} \frac{(u(x)-u(y))}{|x-y|^{N+2s}}(v(x)-v(y))\, dx dy=\lambda \int_{\Omega} u v \, dx+\int_{\Omega} f(x, u) v\, dx
$$
for any $v\in \X$.
For this reason, we will look for critical points of the Euler-Lagrange functional $I_{\lambda}: \X\rightarrow \R$ defined by
\begin{equation}
I_{\lambda}(u)=\frac{1}{2} \|u\|_{\X}^{2}-\frac{\lambda}{2}\int_{\Omega} u^{2}\, dx-\int_{\Omega} F(x, u) \,dx.
\end{equation}
Since we will proceed as in Section $2$, we prove some technical lemmas which will be fundamental to deduce Theorem \ref{thm1}. With suitable modifications, it is easy to see that the next lemma can be proved following the lines of the proof of Lemma \ref{lem1s}.
\begin{lem}\label{lem1}
Assume that $(f1)$ and $(f4)$ hold. Then, for any $\delta\in (0, \min\{\lambda_{i+1}-\lambda_{i}, \lambda_{i}-\lambda_{i-1}\})$ there exists $\e_{0}>0$ such that for any $\lambda\in [\lambda_{i}-\delta, \lambda_{i}+\delta]$ the unique critical point $u$ of $I_{\lambda}$ constrained on $H_{i-1}\oplus H_{i}^{\perp}$ such that $I_{\lambda}(u)\in [-\e_{0}, \e_{0}]$ is the trivial one.
\end{lem}

\begin{lem}\label{lem2}
Assume that $(f1)$ and $(f4)$ hold, $\lambda\in (\lambda_{i-1}, \lambda_{i+1})$ and $\{u_{n}\}_{n\in \N}\subset \X$ such that $I_{\lambda}(u_{n})$ is bounded, $Pu_{n}\rightarrow 0$ and $Q \nabla I_{\lambda}(u_{n})\rightarrow 0$ as $n\rightarrow +\infty$. Then $\{u_{n}\}_{n\in \N}$ is bounded in $\X$.
\end{lem}

\begin{proof}
Suppose by contradiction that, up to a subsequence, $\|u_{n}\|_{\X}\rightarrow \infty$ as $n\rightarrow \infty$. \\
Set $u_{n}=Pu_{n}+Qu_{n}$.
% and we recall that $Pu_{n}\rightarrow 0$ in $\X$. Then, by exploiting the assumption $(f2)$, 
%H\"older inequality and the fact that all norms in $H_{i}^{0}$ are equivalent, we can see that
By using $(f1)$, H\"older's inequality and the fact that all norms in $H_{i}^{0}$ are equivalent, we get
\begin{align}\label{tv4}
\left|\int_{\Omega} f(x, u_{n}) Pu_{n}\, dx \right|&\leq \int_{\Omega} |f(x, u_{n})| |Pu_{n}|\, dx \nonumber \\
&\leq c_{1} \left( \int_{\Omega} |Pu_{n}|\, dx + \int_{\Omega} |Pu_{n}| |u_{n}|^{p}\, dx \right) \nonumber \\
&\leq c_{1} \|Pu_{n}\|_{L^{1}(\Omega)} + c_{1} \left( \int_{\Omega} |u_{n}|^{\beta} dx \right)^{\frac{p}{\beta}}\left(\int_{\Omega} |Pu_{n}|^{\frac{\beta}{\beta-p}}dx \right)^{\frac{\beta-p}{\beta}}\nonumber \\
&\leq c_{5} \|Pu_{n}\|_{L^{\infty}(\Omega)} (1+ \|u_{n}\|_{L^{\beta}(\Omega)}^{p}), 
\end{align}
with $c_{5}>0$. Now, we observe that
\begin{align*}
\langle Q\nabla I_{\lambda}(u_{n}), u_{n}\rangle&=\langle \nabla I_{\lambda}(u_{n}), u_{n}\rangle-\langle P\nabla I_{\lambda}(u_{n}), u_{n}\rangle \nonumber\\
&=\|u_{n}\|_{\X}^{2}-\lambda \|u_{n}\|^{2}_{L^{2}(\Omega)}-\int_{\Omega} f(x, u_{n})u_{n}\, dx \nonumber\\
&-\langle P(u_{n}-(-(-\Delta)^{s})^{-1}(\lambda u_{n}+f(x, u_{n}))), u_{n}\rangle.
\end{align*}
Since $\langle Pu, v\rangle_{\X}=\langle u, Pv\rangle_{\X}$ for any $u, v\in \X$, we can see that
\begin{align*}
\langle P(u_{n}-(-(-\Delta)^{s})^{-1}(\lambda u_{n}+f(x, u_{n}))), u_{n}\rangle&=\|Pu_{n}\|_{\X}^{2}-\lambda \langle Pu_{n}, (-(-\Delta)^{s})^{-1} u_{n}\rangle-\langle Pu_{n}, (-(-\Delta)^{s})^{-1} f(x, u_{n})\rangle \nonumber \\
&=\|Pu_{n}\|{\X}^{2}-\lambda \|Pu_{n}\|_{L^{2}(\Omega)}^{2}-\int_{\Omega} f(x, u_{n})Pu_{n}\, dx.
\end{align*}
Thus $(f4)$ and \eqref{tv4} give
\begin{align}\label{tv5}
&2I_{\lambda}(u_{n})- \langle Q\nabla I_{\lambda}(u_{n}), u_{n}\rangle \nonumber \\
&= \int_{\Omega} (f(x, u_{n}) u_{n} - 2F(x, u_{n}))\, dx + \|Pu_{n}\|{\X}^{2}-\lambda \|Pu_{n}\|_{L^{2}(\Omega)}^{2} - \int_{\Omega} f(x, u_{n}) Pu_{n} \, dx \nonumber \\
&\geq c_{2} \|u_{n}\|_{L^{\beta}(\Omega)}^{\beta} + \|Pu_{n}\|_{\X}^{2} -\lambda \|Pu_{n}\|_{L^{2}(\Omega)}^{2} - c_{5} \|Pu_{n}\|_{L^{\infty}(\Omega)} (1+ \|u_{n}\|_{L^{\beta}(\Omega)}^{p}).
\end{align}
Since $1<p<\beta$, ${\rm dim}H_{i}^{0}<+\infty$ and $\|Pu_{n}\|_{L^{\infty}(\Omega)} \rightarrow 0$ as $n\rightarrow \infty$, from \eqref{tv5} we can deduce that 
\begin{align}\label{tv60}
\frac{\|u_{n}\|_{L^{\beta}(\Omega)}^{p}}{\|u_{n}\|_{\X}}\rightarrow 0 \mbox{ as } n\rightarrow \infty. 
\end{align}
Set $Qu_{n}= v_{n}+ w_{n} \in H_{i-1}\oplus H_{i}^{\perp}$. By using $(f1)$, Theorem \ref{Sembedding}, \eqref{7} and H\"older's inequality we have
\begin{align*}
\langle Q\nabla I_{\lambda}(u_{n}), -v_{n} \rangle &= \lambda \|v_{n}\|_{L^{2}(\Omega)}^{2} - \|v_{n}\|_{\X}^{2} + \int_{\Omega} f(x, u_{n})v_{n}\, dx\\
&\geq \frac{\lambda- \lambda_{i-1}}{\lambda_{i-1}} \|v_{n}\|_{\X}^{2} - \int_{\Omega} |f(x, u_{n})||v_{n}|\, dx \\
&\geq \frac{\lambda- \lambda_{i-1}}{\lambda_{i-1}}\|v_{n}\|_{\X}^{2} -c_{1} \int_{\Omega} (|u_{n}|^{p} |v_{n}| + |v_{n}|)\, dx \\
&\geq \frac{\lambda- \lambda_{i-1}}{\lambda_{i-1}}\|v_{n}\|_{\X}^{2} -c_{1} \left(\int_{\Omega} |u_{n}|^{\beta}dx\right)^{\frac{p}{\beta}} \left(\int_{\Omega} |v_{n}|^{\frac{\beta}{\beta-p}} dx\right)^{\frac{\beta-p}{\beta}} - c_{1}\|v_{n}\|_{L^{1}(\Omega)} \\
&\geq \frac{\lambda- \lambda_{i-1}}{\lambda_{i-1}}\|v_{n}\|^{2}_{\X} -c_{1} C\|v_{n}\|_{\X} (1+ \|u_{n}\|_{L^{\beta}(\Omega)}^{p}). 
\end{align*}
Arguing as in the proof of Lemma \ref{lem2s}, we can see that
%Therefore, \eqref{tv60} and H\"older inequality implies that
\begin{align}\label{16}
\frac{\|v_{n}\|_{\X}}{\|u_{n}\|_{\X}}\rightarrow 0, \quad
%\end{align}
%\begin{align}\label{17}
\frac{\|w_{n}\|_{\X}}{\|u_{n}\|_{\X}}\rightarrow 0 \mbox{ as } n\rightarrow \infty, 
\end{align}
%We also have 
and
\begin{align}\label{18}
\frac{\|Pu_{n}\|_{\X}}{\|u_{n}\|_{\X}}\rightarrow 0 \mbox{ as } n\rightarrow \infty. 
\end{align}
%Indeed, if \eqref{18} does not hold, then $\frac{\|Pu_{n}\|_{\X}}{\|u_{n}\|_{\X}}\rightarrow \ell\in (0, +\infty)$ and we can see that 
%$$
%0\leftarrow \|Pu_{n}\|_{\X}=\|u_{n}\|_{\X}\frac{\|Pu_{n}\|_{\X}}{\|u_{n}\|_{\X}}\rightarrow \ell \cdot (+\infty)=+\infty
%$$
%which is impossible.
Putting together \eqref{16} and \eqref{18} we can see that
\begin{align*}
1= \frac{\|u_{n}\|_{\X}}{\|u_{n}\|_{\X}}\leq \frac{\|v_{n}\|_{\X}+ \|Pu_{n}\|_{\X} + \|w_{n}\|_{\X}}{\|u_{n}\|_{\X}}\rightarrow 0 \mbox{ as } n\rightarrow \infty, 
\end{align*}
which is impossible. 
\end{proof}

\begin{lem}\label{lem3}
Assume that $(f1)$ and $(f4)$. Then, for any $\delta\in (0, \min\{\lambda_{i+1}-\lambda_{i}, \lambda_{i}-\lambda_{i-1}\})$ there exists $\e_{0}>0$ such that for any $\lambda\in [\lambda_{i}-\delta, \lambda_{i}+\delta]$ and for any $\e_{1}, \e_{2}\in (0, \e_{0})$ with $\e_{1}<\e_{2}$, the condition $(\nabla) (I_{\lambda}, H_{i-1}\oplus H_{i}^{\perp}, \e_{1}, \e_{2})$ holds.
\end{lem}

\begin{proof}
The proof follows the lines of the proof of Lemma \ref{lem3s} replacing Lemma \ref{lem1s},  Lemma \ref{lem2s} and Theorem \ref{SSembedding} by Lemma \ref{lem1},  Lemma \ref{lem2}, and  Theorem \ref{Sembedding} respectively. Moreover, in this case, to prove that $u_{n}$ converges strongly in $\X$, we use that fact that $K=(-(-\Delta_{\R^{N}})^{s})^{-1}: L^{q'}(\Omega)\rightarrow \X$ is compact, with $q\in [1, 2^{*}_{s})$; see Section $2.4$ in \cite{MBMS}.
%Suppose by contradiction that there exists a positive constant $\delta_{0}$ such that for all $\varepsilon_{0}>0$ there is $\lambda\in [\lambda_{i}- \delta_{0}, \lambda_{i}+ \delta_{0}]$ and $\varepsilon_{1}, \varepsilon_{2} \in (0, \varepsilon_{0})$ with $\varepsilon_{1}< \varepsilon_{2}$, the condition $(\nabla)(I_{\lambda}, H_{i-1}\oplus H_{i}^{\perp}, \varepsilon_{1}, \varepsilon_{2})$ does not hold. \\
%Let $\varepsilon_{0}>0$ as in Lemma \ref{lem1}. Then, we can find a sequence $\{u_{n}\}_{n\in \N}\subset \X$ such that ${\rm d}(u_{n}, H_{i-1} \oplus H_{i}^{\perp})\rightarrow 0$, $I_{\lambda}(u_{n})\in (\varepsilon_{1}, \varepsilon_{2})$ and $Q\nabla I_{\lambda}(u_{n})\rightarrow 0$. By Lemma \ref{lem2} we deduce that $\{u_{n}\}_{n\in \N}$ is bounded. Thus, by applying Theorem \ref{Sembedding}, there is a subsequence (still denoted by $u_{n}$) and $u\in \X$ such that $u_{n}\rightharpoonup u$ in $\X$ and $u_{n}\rightarrow u$ in $L^{q}(\Omega)$ for any $q\in [1, 2^{*}_{s})$. Taking into account $(f1)$ and $QI_{\lambda}'(u_{n})\rightarrow 0$, it is standard to check that $u_{n}\rightarrow u$ in $\X$ and $u$ is a critical point of $I_{\lambda}$ constrained on $H_{i-1}\oplus H_{i}^{\perp}$. Hence, in view of  Lemma \ref{lem1}, we can infer that $u=0$. Since $0<\varepsilon_{1}\leq I_{\lambda}(u)$, we obtain a contradiction.  
\end{proof}

\noindent
Now, we define the following sets: for fixed $i, k\in \N$ and $R, \varrho >0$, let
\begin{align*}
&B_{i}(R)= \{u\in H_{i} : \|u\|_{\X}\leq R\}, \\
&T_{i-1, i}(R) = \{u \in H_{i-1}: \|u\|_{\X}\leq R\} \cup \{u \in H_{i} : \|u\|_{\X}=R\},  \\
&S_{k}^{+}(\varrho)= \{u \in H_{k}^{\perp}: \|u\|_{\X}= \varrho\}, \\
&B_{k}^{+}(\varrho)= \{u\in H_{k}^{\perp}: \|u\|_{\X}\leq \varrho\}.
\end{align*}

\begin{lem}\label{lem4}
Assume that $(f1)$-$(f3)$ and $(f5)$ hold. Then, for any $\lambda\in (\lambda_{i-1}, \lambda_{i+1})$, there are $R>\varrho>0$ such that
$$
0=\sup I_{\lambda}(T_{i-1, i}(R))<\inf I_{\lambda}(S_{i-1}^{+}(\varrho)).
$$
\end{lem}

\begin{proof}
By using \eqref{7} and the assumption $(f5)$, for any $u\in H_{i-1}$ and $\lambda\in (\lambda_{i-1}, \lambda_{i})$ we have
\begin{align}\label{19}
I_{\lambda}(u)&= \frac{1}{2} \iint_{\R^{2N}} \frac{|u(x)- u(y)|^{2}}{|x-y|^{N+2s}}\, dxdy - \frac{\lambda}{2} \int_{\Omega} |u|^{2} dx - \int_{\Omega} F(x, u)\, dx \nonumber \\
&\leq \frac{\lambda_{i-1}- \lambda}{2\lambda_{i-1}}\|u\|_{\X}^{2} \leq 0. 
\end{align}
%Taking into account the assumption $(f3)$ and the continuity of $F$, for any $c_{6}>0$ there is $M_{1}>0$ such that 
%\begin{align}\label{tv6}
%F(x, t) \geq \frac{c_{6}}{2}t^{2} - M_{1} \quad \forall (x, t)\in \Omega \times \R.
%\end{align}
%By using \eqref{7} and \eqref{tv6}, for any $u\in H_{i}$ and $\lambda\in (\lambda_{i-1}, \lambda_{i})$ we have
Recalling $(f3)$ and \eqref{7}, for any $u\in H_{i}$ and $\lambda\in (\lambda_{i-1}, \lambda_{i})$ we get
\begin{align}\label{tv7}
I_{\lambda}(u)&= \frac{1}{2} \iint_{\R^{2N}} \frac{|u(x)- u(y)|^{2}}{|x-y|^{N+2s}}\, dxdy - \frac{\lambda}{2} \int_{\Omega} |u|^{2} dx - \int_{\Omega} F(x, u)\, dx  \nonumber \\
&\leq \frac{\lambda_{i}- \lambda}{2\lambda_{i}} \|u\|_{\X}^{2} - \frac{c_{6}}{2} \|u\|_{L^{2}(\Omega)}^{2} + M_{1}|\Omega| \nonumber \\
& \leq \frac{\lambda_{i}- \lambda- c_{6}}{2\lambda_{i}} \|u\|_{\X}^{2} + M_{1}|\Omega|. 
\end{align}
Taking $c_{6}= 2(\lambda_{i}- \lambda)$, from \eqref{tv7} we deduce that 
\begin{align}\label{20}
I_{\lambda}(u)\rightarrow - \infty \mbox{ as } \|u\|_{\X}\rightarrow \infty. 
\end{align}
%Now, we note that $(f1)$ and $(f2)$ imply that for any $\e>0$ there is $C_{\e}>0$ such that 
%\begin{align}\label{tv8}
%F(x, t)\leq \frac{\e}{2} t^{2} + C_{\e} |t|^{p+1} \quad \forall (x, t)\in \Omega \times \R,
%\end{align}
%which gives
%\begin{align}\label{21}
%\left| \int_{\Omega} F(x, u)\, dx \right| \leq \frac{\e}{2} \|u\|_{L^{2}(\Omega)}^{2} + C_{\e} \|u\|_{L^{p+1}(\Omega)}^{p+1} \quad \forall u\in \X. 
%\end{align}
%Thus, from \eqref{21},  we can see that for any $u\in H_{i-1}^{\perp}$
By exploiting $(f1)$ and $(f2)$, we can see that for any $u\in H_{i-1}^{\perp}$
\begin{align}\label{22}
I_{\lambda}(u)&= \frac{1}{2} \iint_{\R^{2N}} \frac{|u(x)- u(y)|^{2}}{|x-y|^{N+2s}}\, dxdy - \frac{\lambda}{2} \int_{\Omega} |u|^{2} dx - \int_{\Omega} F(x, u)\, dx  \nonumber \\
&\geq \frac{\lambda_{i}- \lambda- \e}{2 \lambda_{i}} \|u\|_{\X}^{2} - CC_{\e} \|u\|_{\X}^{p+1}. 
\end{align} 
Choosing $\e= \frac{\lambda_{i}- \lambda}{2}>0$, and by using $\lambda \in (\lambda_{i-1}, \lambda_{i})$, $p+1>2$, \eqref{19}, \eqref{20} and \eqref{22}, we can deduce that there exist  $R>\varrho >0$ such that
\begin{align*}
\sup I_{\lambda} (T_{i-1, i}(R))< \inf I_{\lambda} (S_{i-1}^{+}(\varrho)).
\end{align*}
\end{proof}

\noindent
The next result can be obtained following the proof of Lemma \ref{lem5s}.
\begin{lem}\label{lem5}
Assume that $(f5)$ holds. Then, for $R>0$ in Lemma \ref{lem4} and for any $\e>0$ there exists $\delta'_{i}>0$ such that for any $\lambda\in (\lambda_{i}-\delta'_{i}, \lambda_{i})$ it holds
$$
\sup I_{\lambda}(B_{i}(R))<\e.
$$
\end{lem}

%\begin{proof}
%By using \eqref{7}, the assumption $(f5)$ and $\lambda<\lambda_{i}$, for any $u\in H_{i}$ we deduce
%\begin{align*}
%I_{\lambda}(u) = \frac{1}{2} \iint_{\R^{2N}} \frac{|u(x)- u(y)|^{2}}{|x-y|^{N+2s}}\, dxdy - \frac{\lambda}{2} \int_{\Omega} |u|^{2} dx - \int_{\Omega} F(x, u)\, dx \leq \frac{\lambda_{i}- \lambda}{2 \lambda_{i}} \|u\|^{2}_{\X}. 
%\end{align*}
%Take $\delta_{i}'= \frac{\lambda_{i}\e}{R^{2}}$, then we deduce that 
%\begin{align*}
%\sup I_{\lambda}(B_{i}(R)) \leq \frac{\lambda_{i}- \lambda}{2 \lambda_{i}} R^{2} = \frac{(\lambda_{i}- \lambda)\e}{2 \delta_{i}'}<\e. 
%\end{align*}
%\end{proof}

\begin{lem}\label{lem6}
Assume that $(f1)$ and $(f4)$ hold. Then $I_{\lambda}$ verifies the Palais-Smale condition.
\end{lem}
\begin{proof}
Let $\{u_{n}\}_{n\in \N}$ be a Palais-Smale sequence of $I_{\lambda}$.  
%Taking into account $(f1)$, we have only to prove that $\{u_{n}\}_{n\in \N}$ is bounded. From the arguments in Lemma \ref{lem2}, it is enough to prove that 
We have only to show that
\begin{equation}\label{tv8}
\frac{\|Pu_{n}\|_{\X}}{\|u_{n}\|_{\X}}\rightarrow 0 \mbox{ as } n\rightarrow +\infty.
\end{equation}
%In view of $(f4)$, we know that there exist $c_{7}, c_{8}>0$ such that 
%\begin{align}
%f(x, t)t-2F(x, t)\geq c_{7}|t|-c_{8} \mbox{ for any } (x, t)\in \Omega\times \R.
%\end{align}
By using $(f4)$ and the equivalence of the norms on the finite-dimensional space, we get
\begin{align}\label{tv9}
2I_{\lambda}(u_{n})-\langle I'_{\lambda}(u_{n}), u_{n}\rangle&=\int_{\Omega} (f(x, u_{n})u_{n}-2F(x, u_{n})) \, dx \nonumber\\
&\geq \int_{\Omega} (c_{7}|u_{n}|-c_{8})\, dx \nonumber\\
&\geq \int_{\Omega} (c_{7}|Pu_{n}|-c_{7}|v_{n}|-c_{7}|w_{n}|-c_{8})\, dx \nonumber\\
&\geq c_{9}\|Pu_{n}\|_{\X}-c_{10}(\|v_{n}\|_{\X}+\|w_{n}\|_{\X}+1).
\end{align} 
Putting together \eqref{16} and \eqref{tv9}, we can deduce that \eqref{tv8} holds.
\end{proof}

\begin{proof}[Proof of Theorem \ref{thm1}]
%Firstly, we prove the existence of two critical points.
In view of Lemma \ref{lem3}, Lemma \ref{lem4} and Lemma \ref{lem5}, we can take 
$$
a\in (0, \inf I_{\lambda}(S^{+}_{i-1}(\varrho))) \mbox{ and } b>\sup I_{\lambda}(B_{i}(R))
$$
such that $0<a<b<\e_{0}$. Thus the condition $(\nabla)(I_{\lambda}, H_{i-1}\oplus H_{i}^{\perp}, a, b)$ holds. By using Lemma \ref{lem6} and Theorem \ref{MS}, we can find two critical points $u_{1}, u_{2}\in \X$ such that $I_{\lambda}(u_{i})\in [a, b]$ for $i=1, 2$.
The existence of a third critical point will be obtained by applying the Linking Theorem. 
%Taking into account Theorem $5.3$ in \cite{Rab} and Lemma \ref{lem6}, it is enough to prove that 
%there are 
We prove that there are $\delta''_{i}>0$ and $R_{1}>\varrho_{1}>0$ such that for any $\lambda\in (\lambda_{i}-\delta''_{i}, \lambda_{i})$ it results
\begin{equation}\label{23}
\sup I_{\lambda}(T_{i, i+1}(R_{1}))<\inf I_{\lambda}(S^{+}(\varrho_{1})).
\end{equation}  
By using \eqref{8}, $(f1)$, and $(f2)$ we get 
\begin{align}\label{tv10}
I_{\lambda}(u)&=\frac{1}{2} \iint_{\R^{2N}} \frac{|u(x)-u(y)|^{2}}{|x-y|^{N+2s}} \, dx dy-\frac{\lambda}{2} \int_{\Omega} u^{2} \, dx-\int_{\Omega} F(x, u)\, dx \nonumber \\
&\geq \frac{\lambda_{i+1}-\lambda-\e}{2\lambda_{i+1}} \|u\|_{\X}^{2}-CC_{\e}\|u\|_{\X}^{p+1} \mbox{ for any } u\in H_{i}^{\perp}.
\end{align}
Then, taking $\e=\frac{\lambda_{i+1}-\lambda}{2}$, and recalling that $p>1$, from \eqref{tv10} it follows that there are $\varrho_{1}>0$ and $\alpha>0$ such that
\begin{equation}\label{24}
\inf I_{\lambda}(S^{+}_{i}(\varrho_{1}))\geq \alpha>0.
\end{equation}
On the other hand, by using \eqref{7} and $(f5)$, we deduce that
\begin{align}\label{tv11}
I_{\lambda}(u)&=\frac{1}{2} \iint_{\R^{2N}} \frac{|u(x)-u(y)|^{2}}{|x-y|^{N+2s}} \, dx dy-\frac{\lambda}{2} \int_{\Omega} u^{2} \, dx-\int_{\Omega} F(x, u)\, dx \nonumber \\
&\leq \frac{\lambda_{i}-\lambda}{2\lambda_{i}} \|u\|_{\X}^{2} \mbox{ for any } u\in H_{i}.
\end{align}
Therefore \eqref{tv11} implies that there exist $\delta''_{i}>0$ and $R_{1}>0$ such that for any $\lambda\in (\lambda_{i}-\delta''_{i}, \lambda_{i})$ we get
\begin{equation}\label{25}
I_{\lambda}(u)< \alpha \mbox{ for any } \|u\|_{\X}\leq R_{1}.
\end{equation}
Thus by using \eqref{7} and $(f5)$, we can see that for any $u\in H_{i+1}$ and $\lambda\in (\lambda_{i}-\delta''_{i}, \lambda_{i})$, we have 
\begin{align}\label{26}
I_{\lambda}(u)&=\frac{1}{2} \iint_{\R^{2N}} \frac{|u(x)-u(y)|^{2}}{|x-y|^{N+2s}} \, dx dy-\frac{\lambda}{2} \int_{\Omega} u^{2} \, dx-\int_{\Omega} F(x, u)\, dx \nonumber \\
&\leq \frac{\lambda_{i+1}-\lambda}{2\lambda_{i+1}} \|u\|_{\X}^{2}.
\end{align}
Putting together \eqref{24}, \eqref{25} and \eqref{26} we can deduce that \eqref{23} is verified. By applying the Linking Theorem, we can find a third critical point $u_{3}\in \X$ of $I_{\lambda}$ such that $I_{\lambda}(u)\geq \inf I_{\lambda}(S^{+}_{i}(\varrho_{1}))$.
Choosing $\delta_{i}=\min\{ \delta'_{i}, \delta''_{i} \}$, where $\delta'_{i}$ is given in Lemma \ref{lem5}, we can conclude that Theorem \ref{thm1} holds.
\end{proof}

%\begin{remark}
%It is easy to see that Theorem \ref{thm1} holds if we replace the fractional Laplacian operator by the more general integro-differential operator $-\mathcal{L}_{K}$ defined up to a positive constant as
%$$
%\mathcal{L}_{\mathcal{K}} u(x)=\int_{\R^{N}}  (u(x+y)+u(x-y)-2u(x)) \mathcal{K}(y) dy \quad (x\in \R^{N}),
%$$ 
%where $\mathcal{K}:\R^{N}\setminus\{0\}\rightarrow (0, \infty)$ is a measurable function such that $\mathcal{K}(-x)=\mathcal{K}(x)$ for all $x\in \R^{N}\setminus\{0\}$, $m \mathcal{K}\in L^{1}(\R^{N})$ with $m(x)=\min\{|x|^{2}, 1\}$, and there exists $\theta>0$ such that $\mathcal{K}(x)\geq \theta |x|^{-(N+2s)}$ for all $x\in \R^{N}\setminus\{0\}$.
%\end{remark}

\smallskip

\noindent
{\bf Acknowledgements.} 
The author warmly thanks the anonymous referee for her/his useful and nice comments on the paper. 
The manuscript was carried out under the auspices of the INDAM - Gnampa Project 2017 titled:{\it Teoria e modelli per problemi non locali}.

\end{document}